\documentclass[11pt,a4paper,twoside]{amsart}
\usepackage[T1]{fontenc}
\usepackage{mathtools,microtype}
\usepackage{txfonts}
\usepackage[utf8]{inputenc}
\usepackage{amsthm}
\usepackage{amsmath}
\usepackage{amssymb}
\usepackage{enumitem}
\usepackage{xcolor}
\usepackage{comment}
\usepackage{tikz}
\usepackage[colorlinks,citecolor=blue,linkcolor=blue,urlcolor=blue]{hyperref}
\usepackage{indentfirst}

\makeatletter
\renewcommand{\subsection}{\@startsection
  {subsection}{2}{\z@}%
  {-3.25ex\@plus -1ex \@minus -.2ex}%
  {1.2ex \@plus .2ex}%
  {\normalfont\bfseries}}
\makeatother

\makeatletter
\renewcommand{\subsubsection}{\@startsection
  {subsubsection}{3}{\z@}%
  {-3.25ex\@plus -1ex \@minus -.2ex}% espacio antes
  {1.2ex \@plus .2ex}% espacio después
  {\normalfont\bfseries}}
\makeatother

\usetikzlibrary{arrows.meta}

\theoremstyle{plain}
\newtheorem{theorem}{\bfseries Theorem}[section]
\newtheorem{lemma}[theorem]{\bfseries Lemma}
\newtheorem{proposition}[theorem]{\bfseries Proposition}
\newtheorem{corollary}[theorem]{\bfseries Corollary}

\theoremstyle{definition}
\newtheorem{definition}[theorem]{\bfseries Definition}
\newtheorem{example}[theorem]{\bfseries Example}
\newtheorem{remark}[theorem]{\bfseries Remark}

\newcommand{\rint}[0]{\mathcal{O}}
\newcommand{\facets}[0]{\mathcal{F}^1(\sigma)}
\newcommand{\kfacets}[0]{\mathcal{F}^k(\sigma)}

\newcommand{\nfacets}[0]{\mathcal{F}_\mathrm{ns}^1(\sigma)}
\newcommand{\rcells}[0]{\Sigma_{d(n)}}
\newcommand{\rfacets}[0]{\Sigma_{d(n)-1}}
\newcommand{\nrfacets}[0]{\Sigma^\mathrm{ns}_{d(n)-1}}
\newcommand{\perfect}[0]{\mathcal{P}_}
\newcommand{\GL}[0]{\operatorname{GL}_n(\mathbb{Z})}
\newcommand{\SL}[0]{\operatorname{SL}_n(\mathbb{Z})}
\newcommand{\sfacets}[0]{\mathcal{F}_\mathrm{s}^1(\sigma)}
\newcommand{\srfacets}[0]{\Sigma^\mathrm{s}_{d(n)-1}}
\newcommand{\starsrfacets}[0]{\Sigma^{\mathrm{s},*}_{d(n)-1}}
\newcommand{\starnrfacets}[0]{\Sigma^{\mathrm{ns},*}_{d(n)-1}}

\newcommand{\GLo}[0]{\operatorname{GL}_n({\mathcal{O}})}
\newcommand{\SLo}[0]{\operatorname{SL}_n({\mathcal{O}})}
\newcommand{\Her}[0]{\mathcal{H}_n(K)}

\newcommand{\im}{\mathrm{Im}}
\renewcommand{\ker}{\mathrm{Ker}}

\renewcommand{\le}{\leqslant}
\renewcommand{\ge}{\geqslant}
\renewcommand{\leq}{\leqslant}
\renewcommand{\geq}{\geqslant}

\title[{\SMALL Explicit Cycle  for the Cohomology of $\SL$ through Voronoi Complex}]{Explicit canonical cycle at the virtual cohomological dimension of
$\SL$ through Voronoi complex}

\author{Alejandro de la Torre Durán}

\address{Univ. Grenoble Alpes, CNRS, Institut Fourier, F-38000 Grenoble, France}
\email{alejandro.de-la-torre-duran@univ-grenoble-alpes.fr}


\begin{document}
\maketitle

\begin{abstract}
%%PEV
We construct an explicit canonical cycle in the top-dimensional homology of the Voronoi complex associated with an arithmetic group. This cycle relates to the cohomology of $\SL$ with rational coefficients at the virtual cohomological dimension.
This cycle has been previously identified in computational works and conjectured to provide an intrinsic generator. 
Our approach relies on a geometric rigidity property of Voronoi tessellations. Furthermore, an abstract framework for polyhedral tessellations of convex cones under group actions is established, elucidating the underlying mechanism of the construction of such  cycles.
%%PEV
%We study the cohomology of $\SL$ and, more generally, of arithmetic groups with rational coefficients at the virtual cohomological dimension. Using the Voronoi complex associated with reduction theory of quadratic forms, we construct an explicit canonical cycle that generates the top-dimensional homology. This cycle, defined as a weighted sum over top-dimensional Voronoi cells, was previously observed in computational works and conjectured to provide a generator in general. We prove that it is always non-trivial and generates the top homology group for a large class of arithmetic groups, including finite-index subgroups of $\SLo$ and $\GLo$ in both the Euclidean and Hermitian settings. Our approach relies on a geometric rigidity property of Voronoi tessellations. In addition, we establish an abstract framework for polyhedral tessellations of convex cones under group actions, which explains the general mechanism behind the construction of such canonical cycles.
\end{abstract}

\section*{Introduction}

The cohomology of $\SL$, and more generally of arithmetic groups, with
rational coefficients is a cornerstone
of numerous problems arising from number theory, geometry and motivic
cohomology \cite{Harder2025,Venkatesh2016,VenkateshICM2018}.
Providing explicit cohomology classes in the unstable range is often
difficult, even when the cohomology group is a one-dimensional
$\mathbb{Q}$-vector space.
In this paper we give a description of the cohomology of $\SL$ with
rational coefficients at the virtual cohomological dimension
in terms of an explicit canonical (and non-trivial) cycle from the
associated Voronoi complex.

The Voronoi reduction theory provides a rigid polyhedral tessellation of
the cone
of symmetric quadratic forms over $\mathbb{R}$ which induces a cellular
decomposition of $X_n^*$, the space of quadratic forms whose
kernel is defined over $\mathbb{Q}$. The group $\SL$ acts cellularly on $X_n^*$, and
this action can be used to compute the
equivariant homology of $X_n^*$ modulo its boundary $\partial X_n^*$.
This equivariant homology is
isomorphic to $H_q(\SL, \mathrm{St}_n)$, where $\mathrm{St}_n$ denotes the Steinberg module,
defined as
the top reduced homology $\widetilde H_{n-2}(T_n,\mathbb{Z})$ of the spherical Tits building
(cf. \cite{Soule2000}).  By Borel–Serre duality \cite{borel}, the groups $H_\bullet(\SL, \mathrm{St}_n)$ are dual, up to torsion, to the cohomology groups $H^\bullet(\SL, \mathbb{Z})$. 
%%PEV: unnecessary, so I take it back ;)
%This  also works for $\GL$ (and of more
%general arithmetic groups).

In \cite{2002, philippe_advances}, the authors define the \emph{Voronoi complex}, denoted $\operatorname{Vor}_\Gamma = (V_*, d_*)$,  which allows one to compute explicitly $H_\bullet(\Gamma, X_n^*, \partial X_n^*)$, where $\Gamma$ is a finite-index subgroup of $\GL$. The group $V_{d(n)-k}$ is generated by a set of representatives of the $\Gamma$-orbits of codimension-$k$ faces of the top cells of the Voronoi tessellation whose interiors do not intersect $\partial X_n^*$ and whose stabilizers do not invert their orientation.

This approach has been extensively used to compute the cohomology of arithmetic groups. In \cite{2002, philippe_advances}, the cohomology of $\GL$ and $\SL$ with trivial coefficients was computed and led to computations of $K_n(\mathbb{Z})$ for
$n \in \{5, 6, 7\}$. In \cite{K8}, partial results for $\GL$ with $n\in\{8,9,10,11\}$ were given as well as the proof $K_8(\mathbb{Z})$ is the trivial group.
These techniques were also extended to the study of $\operatorname{GL}_n(\mathcal{O}_D)$ and the algebraic $K$-groups of imaginary quadratic fields \cite{hermitian}.

In \cite{philippe_advances} and \cite{hermitian}, it was also proven by explicit computations that the element
\begin{equation} \label{form_top_cycle} \sum_{\sigma\in\Sigma_{d(n)}} \frac{1}{|\Gamma_\sigma|} \sigma \end{equation}
generates the top homology group 
$H_{d(n)}\big(\operatorname{Vor}_\Gamma \otimes \mathbb{Q})
\cong H^0(\Gamma,\mathbb{Q})$
in the following cases: when $\Gamma = \SL$ with $n \le 7$; when 
$\Gamma = \operatorname{SL}_n(\rint_D)$ for low dimension and some discriminants. In both works it was conjectured that this explicit generator holds in full generality.

A key point is that the class defined in \eqref{form_top_cycle} is canonical, as it is represented by the collection of all perfect forms of rank $n$ and their stabilizers in $\Gamma$.

The main goal of this paper is to prove the following result
by using geometric rigidity of the Voronoi complex.

\begin{theorem}
\label{pre_top_cycle}
Let $n\in\mathbb{N}$, $n>0$, and $\Gamma$ be a finite index subgroup of $\SLo$ in the Euclidean case with $n$ even, of $\GLo$ in the Euclidean case with $n$ odd and in the Hermitian case for any $n$. Then $H_{d(n)}\big(\operatorname{Vor}_\Gamma \otimes \mathbb{Q}\big) \cong \mathbb{Q}$ and the formula in \eqref{form_top_cycle} gives a canonical non trivial cycle.
\end{theorem}

The mechanism behind the proof of Theorem~\ref{pre_top_cycle} relies on the fact that codimension-one cells naturally split into two cases: non-self-intersecting facets, that is, faces shared by two different perfect forms, and self-intersecting ones, which belong to a single top cell. In the first case, the two top cells induce opposite orientations on the shared facet, leading to pairwise cancellation after weighting each top cell by the inverse of the order of its stabilizer. In the second case, the cancellation is internal: the $\Gamma$-orbit of the facet splits into two distinct orbits under the stabilizer of the cell, and these contributions come with opposite signs. Finally, the connectedness of the Voronoi graph ensures that all perfect forms must be taken into account.

This mechanism is not specific to the Voronoi complex and can be formulated in a more general setting, which we now describe.

Let $C \subset \mathbb{R}^m$ be an open convex cone such that $\mathrm{cl}(C)$ contains no lines.
Let $\mathcal{T}$ be a locally finite tessellation of $C$ such that each
$\sigma \in \mathcal{T}$ is a top-dimensional polyhedral cone contained in
$\operatorname{cl}(C)$, and every codimension-$1$ face is shared by exactly two
tiles. Suppose that a group $\Gamma$ acts linearly on $\mathbb{R}^m$,
preserving orientation and the tessellation
$\mathcal{T}$, with $\mathcal{T}/\Gamma$ finite. Assume moreover that, for every
$\sigma \in \mathcal{T}$, the stabilizer $\Gamma_\sigma$ is finite. Consider
$\Sigma^{\mathcal T}_m$ and $\Sigma^{\mathcal T}_{m-1}$ as sets of
$\Gamma$-representatives of codimension-$0$ and codimension-$1$ faces,
respectively, whose stabilizers do not reverse their orientation.

Consider
\begin{equation}
\label{general_eq}
    \sum_{\sigma\in\mathcal{S}\subset\Sigma_{m}^{\mathcal{T}}} \lambda_\sigma \left(\sum_{\tau\in\Sigma_{m-1}^{\mathcal{T}}}[\sigma:\tau]\tau\right)
\end{equation}
where $\lambda_\sigma\in\mathbb{Q}$ and the signs $[\sigma:\tau]$ are defined as in \ref{s2.2}.
\begin{theorem}
\label{general}
In the previous setting, it holds that $\eqref{general_eq} = 0$ if and only if $\mathcal{S}=\Sigma_{m}^{\mathcal{T}}$ and $\lambda_\sigma = \frac{\lambda}{|\Gamma_\sigma|}$ for every $\sigma\in\mathcal{S}$.  
\end{theorem}

Theorem~\ref{general} establishes the full generality of this explicit cycle. In particular, Theorem~\ref{pre_top_cycle} holds in the more general setting described in \cite{general_setting}.

Nevertheless, in the present work we restrict ourselves to the Voronoi setting in order to prioritize clarity and brevity of exposition, as the proof of the general theorem proceeds in exactly the same way as in this case.

\section{Preliminaries}
\label{s.1}
\subsection{Euclidean and Hermitian Perfect Forms}
In this section, we introduce the theory of \textit{perfect} forms, which provides the background needed for the construction of the Voronoi complex. We use as reference the book by Martinet \cite{euclidean_lattices_martinet}, which draws on the seminal contributions of Voronoi \cite{Voronoi1908}.

Let $n \in \mathbb{N}$. Throughout the paper, we assume $n>0$. We consider a number field $F$ with ring of integers $\mathcal{O}$, and a field $K$ containing an embedding of $F$. If $F=\mathbb{Q}$, we refer to the Euclidean case and fix $K=\mathbb{R}$.
If $F$ is an imaginary quadratic field, we refer to the Hermitian case and fix $K=\mathbb{C}$.

Denote by $\Her$ the space of $n \times n$ Hermitian matrices with entries in $K$.
This is a real vector space of dimension $\frac{n(n+1)}{2}$ in the Euclidean case and $n^2$ in the Hermitian one. We define

\[
d(n) :=
\begin{cases}
\displaystyle \frac{n(n+1)}{2} - 1, & \text{in the Euclidean case},\\[6pt]
\displaystyle n^2 - 1, & \text{in the Hermitian case}.
\end{cases}
\]

\begin{definition}
\label{def_varias}
Let $C_n^* \subset \Her$ denote the cone of non negative definite hermitian forms in $n$ variables whose kernel is spanned by a proper linear subspace of $F^n$. We let $C_n \subset C^*_n$ be the cone of positive definite forms and $\partial C_n^* = C_n^*\setminus C_n$. Let $h \in C_n$. We define the following notation:
\begin{enumerate}[label=(\roman*)]
    \item the \textit{minimum} of $h$ by
    \[
        \mu(h) := \min_{x \in \rint^n \setminus \{0\}} h(x).
    \]
    \item the set of \textit{minimal vectors} of $h$ by
    \[
        m(h) := \{ x \in \rint^n \setminus \{0\} \mid h(x) = \mu(h) \}.
    \]
    \item We say that $h$ is \textit{perfect} it if the $\mathbb{R}-$span of the set
    \[
        \widehat{m}(h)=\{ xx^* \mid x \in m(h) \}
    \]
    is $\Her$.
    \item  We denote by $\perfect n \subset C_n$ the subset of perfect forms, considered up to homothety.
\end{enumerate}
\end{definition}

\begin{remark}
Definition~\ref{def_varias}(iii) can equivalently be stated as follows: 
a form $h \in C_n$ is perfect if and only if it is uniquely determined, up to homothety, by its set of minimal vectors.
\end{remark}

We fix the inner product
\[
\begin{aligned}
\langle \cdot , \cdot \rangle \colon \Her \times \Her &\longrightarrow \mathbb{R} \\
(A,B) &\longmapsto \operatorname{tr}(AB).
\end{aligned}
\]
Note that, given $x\in \rint^n$ and $Q\in\Her$ it holds 
\begin{equation*}
    \langle Q, xx^*\rangle = Q(x).
\end{equation*}

\begin{comment}
We introduce some classical examples of families of perfect forms in the Euclidean case that come from root lattices:
\begin{example}[Root lattices]
\begin{enumerate}[label = (\roman*)]
    \item Consider the hyperplane of $\mathbb{R}^{n+1}$ defined by the equation $\sum_{i=1}^{n+1} x_i = 0$. Then $\mathbb{A}_n$ is the lattice given by $\mathbb{Z}^{n+1}\cap H$.
    \item We denote by $\mathbb{D}_n$ the set of vectors in $\mathbb{Z}^n$ whose coordinates have even sum.
    \item 
\end{enumerate}
\end{example}
We will denote the same way the perfect forms coming from this lattices.
\end{comment}

\begin{definition}
Let $h\in \perfect n$. We define its \textit{Voronoi domain} as the following set
$$\mathcal{D}(h) = \left\{ \sum_{x\in m(h)} \lambda_x xx^* : \lambda_x \geq 0 \text{ for all } x\in m(h) \right\}.$$
\end{definition}

We introduce the following statement which was proven in \cite{Jaquet_rank_7}, as a consequence of Voronoi`s work:
\begin{proposition}
\label{vor_domain_cono_polyhedrico}
Let $h\in\perfect n$. Then $\mathcal{D}(h)$ is a convex polyhedral cone whose extreme rays are precisely the elements of $m(h)$.
\end{proposition}

\begin{definition}
Given $\sigma$, the Voronoi domain of a perfect form. We denote by 
$\kfacets$ the set of \textit{$k$-codimensional faces of $\sigma$}. 
We write $$\mathcal{F}(\sigma)=\bigcup \kfacets$$ for the set of \textit{faces of $\sigma$}.
\end{definition}

The following statement can be found in \cite[Theorem 7.1.12]{euclidean_lattices_martinet}:
\begin{theorem}
\label{pol_cone}
Let $h, h'\in\perfect n$. If $\mathcal{D}(h)$ and $\mathcal{D}(h')$ share an interior point, then $\mathcal{D}(h) = \mathcal{D}(h')$. 
%%PEV: not needed
%%or, what is the same, $h=h'$.
\end{theorem}

\begin{definition}
Let $h, h' \in \perfect n$ be distinct perfect forms. We say that $h$ and $h'$ are \textit{neighbours} if 
$\mathcal{D}(h) \cap \mathcal{D}(h')$ is a common face of codimension one.
\end{definition}

The latter theorem is a direct consequence of \cite[Theorem 7.2.1]{euclidean_lattices_martinet}:
\begin{theorem}
\label{cod_1_dos_pf}
Let $\sigma$ be the Voronoi domain of a perfect form $h$. Then every $\tau \in \facets$ 
corresponds bijectively to a neighbouring perfect form of $h$.
\end{theorem}

\begin{definition}
In dimension $n$, the \textit{Voronoi graph}, denoted $\mathcal{VG}_n$, is the undirected graph whose vertex set is $\perfect n$. Two vertices $h, h' \in \perfect n$ are joined by an edge if they are neighbours.
\end{definition}

The proof of the next statement can be found in \cite[Theorem 7.4.4]{euclidean_lattices_martinet}:
\begin{theorem}
\label{vor_graph_connected}
The Voronoi graph is connected.
\end{theorem}

Let $\gamma \in \GLo$. We consider the linear action of $\GLo$ on $\Her$ given by
\begin{equation}
\label{action}
    \begin{aligned}
\tilde{\gamma} \colon \Her &\to \Her \\
X&\mapsto \gamma^*  X \gamma
\end{aligned}.
\end{equation}
In particular, $\GLo$ acts over $C_n$ and $C_n^*$. Given $h \in \perfect n$, this action translates at the level of minimal vectors as
\[
m(\tilde{\gamma}(h)) = \gamma^{-1} m(h).
\]
And at the level of the Voronoi's domain,
\begin{equation}
\label{real_action}
\mathcal{D}(\tilde{\gamma}(h)) = \gamma^{-1} \mathcal{D}(h) (\gamma^{-1})^*= \widetilde{(\gamma^{-1})^*}\left(\mathcal{D}(h)\right).
\end{equation}

In the remainder of the document, for $\gamma \in \operatorname{GL}_n(K)$, we denote the linear action defined in \eqref{real_action} by
\begin{equation}
\label{def_act}
    \gamma \cdot \_: \Her \to \Her.
\end{equation}

\begin{proposition}
Consider the action of $\operatorname{GL}_n(K)$ defined in \eqref{def_act}. 
Then this action preserves the orientation of the real vector space $\Her$ in the Hermitian case, 
and in the Euclidean case if and only if $n$ is odd.
\label{deter_action}
\begin{proof}
It should be noted that any element $\gamma \in \GLo$ acts on the orientation of $\Her$ as the sign of $\det\left(\widetilde{(\gamma^{-1})^*}\right)$. In the Hermitian case, $\operatorname{GL}_n(\mathbb{C})$ is connected, hence $\GLo$ acts preserving the orientation of $\mathcal{H}_n(\mathbb{C})$. On the other hand, in the Euclidean case it is well known, and easy to verify, that
$\det(\tilde{\gamma}) = \det(\gamma)^{n+1}$. Therefore, the action of the group $\GL$ preserves the orientation of the space
$\mathcal{H}_n(\mathbb{R})$ if and only if $n$ is odd.
\end{proof}
\end{proposition}

The following theorem was originally proved by Voronoi in the Euclidean case \cite{Voronoi1908} and later generalized to algebraic number fields in \cite{Bordeaux}.
\begin{theorem}
\label{vor_th}
There is a finite number of $\GLo-$orbits in $\perfect n$.
\end{theorem}

\subsection{The Voronoi Complex}
\label{s2.2}
In this section, we introduce the Voronoi complex, originally defined in \cite{2002, philippe_advances} for the Euclidean case and generalized to the Hermitian one in \cite{hermitian}.

Consider $\Gamma<\GLo$ a finite index subgroup. Denote by $\mathcal{C}^{d(n)-k}$ a set of representatives of the $\Gamma$-orbits of $\bigcup_{h\in\perfect n}\mathcal{F}^k(\mathcal{D}(h)).$

Set
\[
\Sigma_{d(n)-k}^*(\Gamma)= \{ \sigma \in \mathcal{C}^{d(n)-k} : \operatorname{int}(\sigma) \cap \partial C_n^* = \emptyset \}.
\]

Note that, by definition, it follows that $\Sigma_{k}^*\left(\Gamma\right) = \emptyset$ for $k<n-1$ and $k>d(n)$.

For $\sigma \in \Sigma_k^*(\Gamma)$, we denote by $\widehat{m}(\sigma)$ the set of its extreme rays, and by 
$\mathbb{R}(\sigma)$ the vector subspace of $\mathbb{R}^{d(n)+1}$ generated by $\widehat{m}(\sigma)$. 
Finally, we define
\[
m(\sigma) := \{\, x \in \rint^n : xx^* \in \widehat{m}(\sigma) \,\}.
\]

Let $\sigma \in \Sigma_k^*(\Gamma)$ and $\tau' \in \mathcal{F}^1(\sigma)$ such that there
exist $\tau \in \Sigma_{k-1}^*(\Gamma)$ and $\gamma \in \Gamma$ with $\tau' = \gamma\cdot \tau$.
Fix orientations on $\mathbb{R}(\sigma)$, $\mathbb{R}(\tau)$, and
$\mathbb{R}(\tau')$.

Let $\mathcal{B}'$ be a positively oriented basis of $\mathbb{R}(\tau')$. Then, for
any $v \in m(\sigma) \setminus m(\tau')$, the set $\mathcal{B}' \cup \{v\}$ is a basis
of $\mathbb{R}(\sigma)$, and its orientation does not depend on the choice of $v$.
We define
\[
\varepsilon(\sigma, \tau') =
\begin{cases}
1 & \text{if } \mathcal{B}' \cup \{v\} \text{ is positively oriented}, \\
-1 & \text{otherwise}.
\end{cases}
\]

Given $\mathcal{B}$ a positively oriented basis of $\mathbb{R}(\tau)$, we define
\[
\eta(\tau, \tau') =
\begin{cases}
1 & \text{if } \gamma\cdot \mathcal{B} \text{ is positively oriented in }
\mathbb{R}(\tau'), \\
-1 & \text{otherwise}.
\end{cases}
\]

Finally, let $\sigma \in \Sigma_k^*(\Gamma)$ and $\tau \in \mathcal{F}^1(\sigma)$. Fix an
orientation of $\mathbb{R}(\sigma)$. We call the orientation of
$\mathbb{R}(\tau)$ for which $\varepsilon(\sigma, \tau) = 1$ the \textit{orientation induced
by $\sigma$}.

Given $\sigma \in \Sigma_k^*(\Gamma)$, we denote by
$$\Gamma_\sigma = \{ \gamma \in \Gamma : \gamma\cdot\sigma = \sigma \}$$
\textit{the stabilizer of $\sigma$}. We define $\Sigma_k(\Gamma) \subset \Sigma_k^*(\Gamma)$ as the subset
consisting of those $\sigma$ such that no element of $\Gamma_\sigma$ inverts the
orientation of $\mathbb{R}(\sigma)$.

Given $\sigma\in\Sigma_{k}^*(\Gamma)$ and $\tau\in\Sigma_{k-1}^*(\Gamma)$ we consider $$\mathrm{Orb}_\sigma(\tau) = \{\tau'\in\mathcal{F}^1(\sigma)  \hspace{2pt} | \hspace{2pt}\text{ there exists } \gamma\in \Gamma : \tau' = \gamma\cdot\tau\} $$
and let 
$$[\sigma:\tau] = \sum_{\tau'\in \mathrm{Orb}_\sigma(\tau) } \eta(\tau,\tau')\varepsilon(\sigma,\tau').$$

Set the map
\begin{equation}
\label{differential}
    d_k(\sigma) = \sum_{\tau\in\Sigma_{k-1}(\Gamma)} [\sigma:\tau]\tau.
\end{equation}

%\begin{definition}
Let $V_k$ be the free abelian group generated by $\Sigma_k$, and let
$d_k : V_k \to V_{k-1}$ be defined on the generators by \eqref{differential} for all $k\ge n$.
It has been shown in \cite{2002, philippe_advances} that $d_k$ is a differential and endows
$(V_*, d_*)$ with a structure of cellular complex, called 
%The resulting cellular complex $\operatorname{Vor}^\Gamma_n = (V_k, d_k)$ is called 
the \textit{Voronoi complex of $\Gamma$}, and denoted  $\operatorname{Vor}_\Gamma$.
%\end{definition}

\begin{remark}
From Theorem~\ref{vor_th}, it follows that $\operatorname{Vor}_\Gamma$ is a finite complex. Moreover, as observed above, its top degree is $d(n)$ and its bottom degree is $n-1$. Notice also that due to the geometric nature of the Voronoi domains associated to the perfect forms, this complex is in general not simplicial.
\end{remark}

From here on out, whenever there is no risk of confusion, we will omit the dependence of $\Sigma^*_\bullet(\Gamma)$ and $\Sigma_\bullet(\Gamma)$ on $\Gamma$ and simply write $\Sigma^*_\bullet$ and $\Sigma_\bullet$.

\begin{remark}
\label{dif_sl_gl}
We have $\Sigma_k^*\left(\GLo\right) = \Sigma_k^*\left(\SLo\right)$ if and only if for every $\sigma \in \Sigma_k^*\left(\GLo\right)$ there exists $g \in \Gamma_\sigma$ such that $\det(g) = -1$. In particular, if $n$ is odd, there is no distinction between the Voronoi complex for $\GLo$ and $\SLo$, since $\det(-I_n) = -1$ and $-I_n \in \Gamma_\sigma$ for every $\sigma \in \Sigma^*_\bullet\left(\GLo\right)$.
\end{remark}

\section{The Explicit Generator of the Top Homology of the Voronoi Complex}
\label{s.2}
\subsection{Structural Properties of the Voronoi Complex for a General  \texorpdfstring{$\Gamma$}{Gamma}}

We start by proving some structural properties of the Voronoi Complex that are true without extra hypothesis on the group $\Gamma < \GLo$. Moreover, we introduce the definition of the self and non-self-intersecting facets of a top-cell.

\begin{lemma}
\label{orientaciones_opuestas}
Let $\sigma$ and $\rho$ be two Voronoi domains associated with perfect forms that are neighbors. Let $\mathcal{B}$ be a basis of $\mathbb{R}(\tau)$. Then, for every $v \in m(\sigma) \setminus m(\tau)$ and $v' \in m(\rho) \setminus m(\tau)$, the basis $\mathcal{B} \cup \{v\}$ and $\mathcal{B} \cup \{v'\}$ of $\mathbb{R}^{d(n)+1}$ have opposite orientations.
\end{lemma}
\begin{proof}
Since $\sigma$ is a polyhedral cone (Theorem~\ref{pol_cone}) we have $\mathbb{R}(\tau)\cap\sigma = \tau$. Note that $\mathbb{R}(\tau)$ is a hyperplane in $\Her$. Since $\sigma$ and $\rho$ are neighbours through $\tau$, we have $\rho$ and $\sigma$ lie in different closed half-spaces of $\Her$ defined by $\mathbb{R}(\tau)$. 

Let $N_\tau$ be the unit normal vector to $\mathbb{R}(\tau)$ pointing to the interior of $\sigma$, i.e., such that $\langle N_{\tau}, v\rangle \geq 0$ for every $v \in m(\sigma)$. As $\rho$ lies in the other closed half-space, we have $\langle N_{\tau}, v\rangle \leq 0$ for every $v\in m(\rho)$. Consider $\mathcal{B}$ a positively oriented basis of $\mathbb{R}(\tau)$. As for every $v\in \left(m(\sigma)\cup m(\rho)\right)\setminus m(\tau)$, the orientation of $\mathcal{B}\cup\{v\}$ is determined by the sign of $\langle v, N_\tau\rangle$, we conclude that $\sigma$ and $\rho$ induce opposite orientations on $\tau$.
\end{proof}

\newpage
\begin{lemma}
\label{cod1_no_interseca_frontera}
Let $\sigma\in\rcells^*$. For any $\tau\in\facets$ we have that $\operatorname{int}(\tau) \cap \partial C_n^* = \emptyset$. In other words, $\tau$ is equivalent to an element in $\rfacets^*$.
\end{lemma}
\begin{proof}
We prove the claim by contradiction. Take $\rho \in \rcells^*$ such that $\tau = \sigma \cap \gamma\cdot\rho$ for $\gamma\in\Gamma$. Suppose that there exists $Q\in \operatorname{int}(\tau) \cap \partial C_n^*$. By definition of Voronoi domain, $$Q = \sum_{x\in m(\tau)} \lambda_x xx^*\, ,$$ such that $\lambda_x > 0$ for every $x\in m(h)$. Let $y\in\mathbb{R}^n$ such that $Q(y)=0$. Then $$ \sum_{x\in m(\tau)} \lambda_x xx^*(y) = \sum_{x\in m(\tau)} \lambda_x\langle xx^*,yy^*\rangle = 0.$$
From which we conclude that $yy^*\in \mathbb{R}(\tau)^\perp$. Let $N_\tau$ be an orthogonal vector to $\mathbb{R}(\tau)$. Since $\mathbb{R}(\tau)$ is an hyperplane and $yy^*\notin\mathbb{R}(\tau)$, it holds that $\langle N_\tau, yy^*\rangle = N_\tau(y)\neq 0$. Assume without loss of generality that $N_\tau(y)<0$. For $\delta>0$ small enough, we have $Q' = Q+\delta N_\tau\in\operatorname{int}(\gamma\cdot\rho)$.
But $Q'(y) = \delta N_\tau(y) <0$, which is absurd.
\end{proof}

\begin{remark}
\label{re}
By construction of the Voronoi complex, for any $\tau\in\Sigma^*_{d(n)-1}$, up to changing the choice of the representative, we can assume that $\tau = \sigma \cap \gamma\cdot\rho$ for $\sigma,\rho\in\Sigma^*_{d(n)}$ and $\gamma\in \Gamma$. 
\end{remark}

\begin{definition}
\label{self_inter_def}

Consider $\sigma\in\Sigma_{d(n)}^*$. The set $$\sfacets = \left\{ \tau\in\facets : \text{ there exists } \gamma\in\Gamma \text{ such that } \tau = \sigma \cap \gamma\cdot\sigma\right\}$$ is called the set of \textit{self-intersecting facets of $\sigma$}. 

We call $$\nfacets = \facets\setminus\sfacets$$ the set of \textit{non-self-intersecting facets of $\sigma$}.

We let $$\starsrfacets = \left\{ \tau \in \rfacets^* : \text{ there exists } \sigma \in \rcells^* \text{ such that }\tau \in \sfacets \right\}$$
the set of \textit{ representatives of self-intersecting facets} and  $$\starnrfacets = \rfacets^*\setminus\starsrfacets$$
the set of \textit{ representatives of non-self-intersecting facets}.

\end{definition}

\begin{remark}
\label{atleast_one_nbh}
Note that if $|\mathcal{P}_n/\Gamma|>1$. Then, by connectedness of the Voronoi graph, for $\sigma\in\rcells^*$, we have $\nfacets\neq\emptyset$.
\end{remark}

\begin{lemma}
\label{stab_groups}
Consider $\tau\in\starnrfacets$, and let $\sigma, \rho \in \rcells^*$ such that $\tau = \sigma \cap \gamma\cdot\rho$ for $\gamma\in \Gamma$.
The following holds:

\begin{enumerate}[label=(\roman*)]
\item we have $\Gamma_\tau = \Gamma_{\sigma}\cap\Gamma_{\gamma\cdot\rho}$\,.
\item Consider the action of $\Gamma_\sigma$ on $\facets$. Then, $|\Gamma_\sigma \cdot \tau| = (\Gamma_\sigma : \Gamma_\tau) = \frac{|\Gamma_\sigma|}{|\Gamma_\tau|}$\,.
\end{enumerate}
\begin{proof}
\begin{enumerate}[label=(\roman*)]

\item Let $g \in \Gamma_\tau$, thus $g\cdot\tau = g\cdot\sigma\cap g\cdot(\gamma\cdot\rho) = \tau$. Then, as every facet of a Voronoi domain correspond exactly to two perfect forms, see Theorem~\ref{cod_1_dos_pf}, we have $g$ preserves $\sigma$ and $\rho$ or it commutes them. Since $\sigma$ and $\rho$ are not equivalent under the action by $\Gamma$, we conclude that $g\cdot\sigma = \sigma$ and $g\cdot(\gamma\cdot\rho) = \gamma\cdot\rho$. The other inclusion is straightforward.

\item Consider the map $\phi:\Gamma_\sigma \to \facets$, defined by $\phi(g) = g\cdot\tau$. Note that given $g_1, g_2\in\Gamma_\sigma$ we have $g_1\cdot\tau = g_2\cdot\tau$ if and only if $(g_1g_2^{-1})\cdot\tau = \tau$, or what is the same $g_1g_2^{-1}\in\Gamma_{\tau}$. Note also that $\im(\phi) = \Gamma_\sigma \cdot\tau$. Therefore, as $\Gamma_\tau < \Gamma_\sigma$ we have $\Gamma_\sigma \cdot\tau \cong \frac{\Gamma_\sigma}{\Gamma_\tau}\,.$\qedhere
\end{enumerate}
\end{proof}
\end{lemma}

\subsection{The Voronoi Complex for an Orientation-Preserving $\Gamma$ on $\Her$}
\label{3.2}

In this section we consider $\Gamma < \GLo$ a finite index subgroup whose action on $\Her$ preserves the orientation of the real vector space $\Her$. 
In practice, by Proposition~\ref{deter_action}, this means that $\Gamma$ is a finite index subgroup of $\SLo$ in the Euclidean case with $n$ even, of $\GLo$ in the Euclidean case with $n$ odd and in the Hermitian case for any $n$.

\begin{lemma}
\label{orientation_SLnZ}
 The following equalities hold:
\begin{enumerate}[label=(\roman*)]
\item $\Sigma_{d(n)} = \Sigma_{d(n)}^*$\,.
\item $\nrfacets = \starnrfacets$\,.
\end{enumerate}

\begin{proof}
\begin{enumerate}[label=(\roman*)]
\item It is a direct consequence of Proposition~\ref{deter_action}.
\item Let $\tau\in\starnrfacets$,  $\sigma, \rho\in\rcells$ and $\gamma\in\Gamma$ such that $\tau = \sigma \cap \gamma\cdot\rho$. Suppose that $\tau\notin\nrfacets$. Without loss of generality, we can assume that $\varepsilon(\sigma,\tau) = 1$. By hypothesis, there exists $g\in\Gamma_\tau$ such that $\eta(\tau,g\tau)=-1$. 

Consider $\mathcal{B}$ a positively oriented basis of $\mathbb{R}(\tau)$ and $v\in m(\sigma)\setminus m(\tau)$. Then, $\mathcal{B}\cup\{v\}$ it is positively oriented in $\mathbb{R}(\sigma)$. By Lemma \ref{stab_groups} i) we have $\Gamma_\tau \subset \Gamma_\sigma$, therefore $g\cdot v\in m(\sigma)\setminus m(\tau)$. Thus, the orientation of $g\cdot\mathcal{B}\cup \{ g\cdot v\}$ is the same as the orientation of $g\cdot\mathcal{B}\cup\{v\}$. As the orientation of $\mathcal{B}$ is opposite to the orientation of $g\cdot\mathcal{B}$, we conclude that $g\cdot\mathcal{B}\cup\{g\cdot v\}$ is negatively oriented in $\mathbb{R}(\sigma)$. Then, $g$ inverts the orientation of $\mathbb{R}(\sigma)$, which contradicts (i). \qedhere
\end{enumerate}
\end{proof}
\end{lemma}

\begin{lemma}
\label{self_intersecting_facets_stab}
Let $\tau\in\starsrfacets$ and $\sigma\in\rcells^*$ such that $\tau = \sigma\cap\gamma\cdot\sigma$ for $\gamma\in\Gamma$. Let $$\Gamma_{(\sigma,\gamma\cdot\sigma)} = \left\{ s \in \Gamma : s\cdot\sigma = \gamma\cdot\sigma \text{ and } s\cdot(\gamma\cdot\sigma) = \sigma\right\}.$$
Then, the following holds:
\begin{enumerate}[label=(\roman*)]
    \item $\Gamma_\tau = \left( \Gamma_\sigma \cap \Gamma_{\gamma\cdot\sigma} \right) \sqcup \Gamma_{(\sigma, \gamma\cdot\sigma)}$\,.
    \item $\Gamma_\sigma \cap \Gamma_{\gamma\cdot\sigma} < \Gamma_\tau$ and $\left(\Gamma_\tau : \Gamma_\sigma \cap \Gamma_{\gamma\cdot\sigma}\right) = k$ where $k = 1$ or $k = 2$.
\end{enumerate}

\begin{enumerate}
    \item[(iii)] The following statements are all equivalent:
    \begin{enumerate}[label=(\alph*)]
        \item $\gamma^{-1}\cdot\tau \sim_{\Gamma_\sigma} \tau.$
        \item $\tau\notin\srfacets\,.$
        \item $\left(\Gamma_\tau : \Gamma_\sigma \cap \Gamma_{\gamma\cdot\sigma}\right) = 2.$
    \end{enumerate}
\item[(iv)] Suppose also that $\tau\in\srfacets$. Then the orbit $\gamma\cdot\tau$ splits into exactly two $\Gamma_\sigma$-orbits, namely
$\Gamma_\sigma \cdot \tau$ and $\Gamma_\sigma \cdot (\gamma^{-1}\cdot\tau)$, when restricting to the action
of $\Gamma_\sigma$.
\end{enumerate}
\begin{proof}

\begin{enumerate}[label=(\roman*)]
\item The argument is similar to \ref{stab_groups} (i), but now we are allowed to permute the cells $\sigma$ and $\gamma\cdot\sigma$. The elements $s\in\Gamma$ that permute this two correspond to the set $\Gamma_{(\sigma,\gamma\cdot\sigma)}$.
\item It follows from the fact that $\Gamma_{(\sigma,\gamma\cdot\sigma)}\Gamma_{(\sigma,\gamma\cdot\sigma)} \subset \Gamma_\sigma\cap\Gamma_{\gamma\cdot\sigma}$ and $\Gamma_{(\sigma,\gamma\cdot\sigma)}^{-1}\subset \Gamma_{(\sigma,\gamma\cdot\sigma)}$.

\item (b) $\implies$ (c): By the argument in Lemma~\ref{orientation_SLnZ}, if we have 
$g \in \Gamma_\sigma \cap \Gamma_{\gamma\cdot\sigma} \subset \Gamma_\sigma$
that inverts the orientation of $\mathbb{R}(\tau)$, then $g$ also inverts
the orientation of $\mathbb{R}(\sigma)$. This is impossible, as a consequence of Proposition~\ref{deter_action}.\\
(c) $\implies$ (b): Let $s\in\Gamma_{(\sigma,\gamma\cdot\sigma)}$. Consider $\mathcal{B}$ a positively oriented basis of $\mathbb{R}(\tau)$ and assume without loss of generality that $\varepsilon(\sigma,\tau) = 1$, i.e, given $v\in m(\sigma)\setminus m(\tau)$ we have $\mathcal{B}\cup \{v\}$ is positively oriented in $\mathbb{R}(\sigma)$. By Proposition~\ref{deter_action}, $s\cdot\mathcal{B}\cup \{s\cdot v\}$ is positively oriented in $\mathbb{R}(\sigma)$. By assumption, $s\cdot\sigma = \gamma\cdot\sigma$. Thus, $s\cdot v\in m(\gamma\cdot\sigma)\setminus m(\tau)$. By Lemma~\ref{orientaciones_opuestas}, we conclude that $s\cdot\mathcal{B}$ is negatively oriented in $\mathbb{R}(\tau)$. Thus, $\tau\notin\srfacets$.\\
(a) $\implies$ (c): Consider $g\in\Gamma_\sigma$ such that $g\cdot(\gamma^{-1}\cdot\tau) = \tau$. This gives the following chain of equalities:

\begin{align*}
  \sigma \cap g\cdot(\gamma^{-1}\cdot\sigma) &= g\cdot\sigma \cap g\cdot(\gamma^{-1}\cdot\sigma),\\ 
  &=  g\cdot(\sigma\cap\gamma^{-1}\cdot\sigma),\\
  &=  g\cdot(\gamma^{-1}\cdot\tau),\\
  &=  \tau = \sigma \cap \gamma\cdot\sigma.  
\end{align*}

Then, $g\gamma^{-1}\cdot\sigma = \gamma\cdot\sigma$ and $g\gamma^{-1}\cdot(\gamma\cdot\sigma) = g\cdot\sigma = \sigma$. Thus $g\gamma^{-1}\in\Gamma_{\left(\sigma, \gamma\cdot\sigma \right)}$.

(c) $\implies$ (a) Let $s\in\Gamma_{(\sigma,\gamma\cdot\sigma)}$. Then $$s\gamma\cdot \sigma = s\cdot(\gamma\cdot\sigma)= \sigma.$$ Thus, $s\gamma\in\Gamma_\sigma$. Moreover, $$s\gamma\cdot(\gamma^{-1}\cdot\tau) = s\cdot\tau = \tau.$$
\item By (iii), $\tau\not\sim_{\Gamma_\sigma}\gamma^{-1}\cdot\tau$. We show that no other orbit exists.

Consider $\omega \in \Gamma$ such that $\omega \cdot(\gamma \cdot\sigma) = \sigma$. Equivalently, $\gamma \cdot \sigma = \omega^{-1}\cdot\sigma$,
which implies that $\omega^{-1} \in \gamma \Gamma_\sigma$. Thus, there exists
$g \in \Gamma_\sigma$ such that $\omega^{-1} = \gamma g$, or equivalently,
$\omega = g^{-1}\gamma^{-1}$. Therefore,
\[
\omega\cdot\tau = g^{-1}\gamma^{-1}\cdot\tau,
\]
that is, $\omega \tau \sim_{\Gamma_\sigma} \gamma^{-1}\cdot\tau$.
\qedhere
\end{enumerate}
\end{proof}
\end{lemma}

\begin{example}
\label{examplee}
Consider $\Gamma = \operatorname{GL}_2(\mathbb{Z})$. Let $\varepsilon_i$ be the standard basis of $\mathbb{R}^2$. Consider $\mathbb{A}_2$ the perfect form defined by the minimal vectors $m(\mathbb{A}_2) = \{\pm\varepsilon_1, \pm\varepsilon_2, \pm(\varepsilon_1-\varepsilon_2)\}$. We have that $\sigma_2 =\mathcal{D}(\mathbb{A}_2)$ is a $2-$ dimensional simplex.  

Consider $$\gamma = \begin{pmatrix} 1 & 0 \\ 0 & -1 \end{pmatrix}.$$ 
Let $\tau = \sigma_2\cap\gamma\cdot\sigma_2\in\mathcal{F}^1(\sigma_2)$ defined by $m(\tau) = \{\pm\varepsilon_1,\pm\varepsilon_2\}$. Check \cite[Figure 5]{Jaquet_rank_7}.

 Then, as the stabilizer of $\tau$ coincides with the stabilizer of its barycenter, see \cite[Section 4.1]{philippe_advances}, we have
$$\Gamma_{\tau} = \{ \gamma\in\Gamma : \gamma^t\gamma = I_2 \} = \operatorname{O}_2(\mathbb{Z}).$$ On the other hand, 
$$\Gamma_{\sigma_2}\cap\Gamma_{\gamma\cdot\sigma_2} = \left\{\pm I_2, \pm\begin{pmatrix} 0 & 1 \\ 1 & 0 \end{pmatrix}\right\}.$$ 

Therefore, the decomposition described in Lemma~\ref{self_intersecting_facets_stab} (i) can be expressed as
$$\Gamma_{(\sigma_2,\gamma\cdot\sigma_2)} = \left\{\pm\begin{pmatrix} 1 & 0 \\ 0 & -1 \end{pmatrix}, \pm\begin{pmatrix} 0 & 1 \\ -1 & 0 \end{pmatrix}\right\}.$$
\end{example}

The following result is well known in the literature; we include it as an illustration of Lemma~\ref{self_intersecting_facets_stab}.

\begin{corollary}
\label{cor}
Let $n=2,3$ and $\Gamma = \SL$. Then $\rcells = \{ \mathbb{A}_n \}$ and $\rfacets = \emptyset$.
\begin{proof}
Denote $\sigma_n = \mathcal{D}(\mathbb{A}_n)$, which is a simplex. In \cite[Theorem 7.5.1]{euclidean_lattices_martinet} it is proven that $\GL_{\sigma_n}\cong \mathfrak{S}_{n+1}\times\{\pm I\}$ acts transitively in $\mathcal{F}^1(\sigma_n)$. 

Consider $n=3$. Due to Remark~\ref{dif_sl_gl}, $\Sigma_\bullet(\operatorname{GL}_3(\mathbb{Z})) = \Sigma_\bullet(\operatorname{SL}_3(\mathbb{Z}))$. Then, $\Sigma_5(\operatorname{SL}_3(\mathbb{Z})) = \{\mathbb{A}_3\}$. Applying Lemma~\ref{self_intersecting_facets_stab} iii), we conclude that $\Sigma_4(\operatorname{SL}_3(\mathbb{Z})) = \emptyset$. 

The same argument can be applied for the case $n=2$, following Example~\ref{examplee}.\qedhere
\end{proof}
\end{corollary}

\subsection{
The Explicit Generator of 
\texorpdfstring{$H_{d(n)}( \mathrm{Vor}_\Gamma \otimes \mathbb{Q})$}
{Top homology of VorGamma_n}
}

Fix $\Gamma$ as in Section~\ref{3.2}, that is, $\Gamma$ is a finite index subgroup of $\SLo$ in the Euclidean case with $n$ even, of $\GLo$ in the Euclidean case with $n$ odd and in the Hermitian case for any $n$. 

Consider $\rfacets$ verifying the property described in Remark \ref{re}. For $\sigma\in\Sigma_{d(n)}$ consider the orientation of $\mathbb{R}(\sigma)$ as the usual orientation of $\mathbb{R}^{d(n)+1}$. Given $\tau\in\Sigma_{d(n)-1}$, by Remark \ref{re} we can choose $\sigma\in\Sigma_{d(n)}$ such that $\tau\in\facets$. Consider the orientation in $\mathbb{R}(\tau)$ induced by the orientation of $\sigma$, i.e, an orientation $\mathcal{B}$ such that for any $v\in m(\sigma) \setminus m(\tau)$ we have $\mathcal{B}\cup\{ v\}$ is positively oriented in $\mathbb{R}^{d(n)+1}$.

Let $g\in \Gamma$. We consider the orientation $g\cdot\mathcal{B}$ in $\mathbb{R}(g\cdot\tau)$. Note that this orientation does no depend on the choice of $g$ as $\tau\in\rfacets$. 
This fixes an orientation for every $\mathbb{R}(\tau)$ such that $$\tau\in\bigcup_{\sigma\in\Sigma_{d(n)}} \{ \tau'\in\facets: \tau' \text{ is equivalent to an element in } \rfacets\}.$$

\begin{lemma}
\label{expresion_diferential}
Consider $\tau\in\rfacets$. Let $\sigma\in\rcells$ such that $\tau'= \omega\cdot\tau\in\facets$ for $\omega\in\Gamma$. Then, the following holds:
\begin{enumerate}[label=(\roman*)]
    \item  For every $g\in\Gamma_\sigma$ we have $\varepsilon(\sigma,g\cdot\tau') =  \varepsilon(\sigma,\tau')$ and $\eta(\tau,\tau') = 1$.

    \item Suppose that $\tau\in\nrfacets$. Let $\rho\in\rfacets$ such that $\tau = \sigma\cap\gamma\cdot\rho$. Then, $[\sigma:\tau] = \frac{|\Gamma_\sigma|}{|\Gamma_\tau|}$ and $[\rho:\tau] = -\frac{|\Gamma_\rho|}{|\Gamma_\tau|}$\,.
    
    \item Suppose that $\tau\in\srfacets$. Then, $[\rho:\tau]=0$ for every $\rho\in\rcells$.
\end{enumerate}
\begin{proof}
\begin{enumerate}[label=(\roman*)]
\item The first equality is given by the choice of the orientation and by an analogous argument to the one presented in Lemma \ref{orientation_SLnZ} (ii). For the second equality, note that given  $\mathcal{B}$ a positive orientation of $\mathbb{R}(\tau)$, we let $\gamma\cdot\mathcal{B}$ be a positive orientation of $\mathbb{R}(\gamma\cdot\tau)$, which is equivalent to $\eta(\tau,\gamma\cdot\tau) = 1$.

\item Assume that $\mathbb{R}(\tau)$ was given the orientation inherit from $\sigma$, i.e, $\varepsilon(\sigma,\tau) = 1$.
We are going to compute $[\sigma : \tau]$. Applying (i), we get that 
$$[\sigma:\tau] =  \sum_{\tau'\in \mathrm{Orb}_\sigma(\tau) } 1.$$
By Lemma \ref{stab_groups} (ii), we know that $|\Gamma_{\sigma}\cdot\tau| = \frac{|\Gamma_\sigma|}{|\Gamma_\tau|}$. As, $\tau\in\facets$, we have $\mathrm{Orb}_\sigma(\tau) =  \Gamma_{\sigma}\cdot\tau$. Therefore, $[\sigma:\tau] =\frac{|\Gamma_\sigma|}{|\Gamma_\tau|}$. 

Lets compute now $[\rho : \tau ]$. Take $\mathcal{B}$ a positively oriented basis of $\mathbb{R}(\tau)$. By Lemma \ref{orientaciones_opuestas} we have for $v\in m(\gamma\cdot\rho)\setminus m(\tau)$ the basis $\mathcal{B}\cup\{v\}$ is negatively oriented in $\mathbb{R}^{d(n)+1}$. By Proposition~\ref{deter_action}, we have $\gamma^{-1}$ preserves the orientation of $\mathbb{R}^{d(n)+1}$. Therefore, $\gamma^{-1}\cdot\mathcal{B}\cup\{\gamma^{-1}\cdot v\}$ is negatively oriented in $\mathbb{R}(\rho)$, i.e., $\varepsilon(\rho,\gamma^{-1}\cdot\tau) = -1$. Applying (i) and Lemma \ref{stab_groups} as above, we get that $[\rho:\tau] =-\frac{|\Gamma_\rho|}{|\Gamma_\tau|}$. 

\item By definition, $[\rho: \tau] = 0$ for $\rho\in\rfacets\setminus\{\sigma\}$. Using Lemma~\ref{self_intersecting_facets_stab} (iv), we know that $\Gamma_\tau$ splits into exactly two $\Gamma_\sigma$-orbits, namely
$\Gamma_\sigma \cdot \tau$ and $\Gamma_\sigma\cdot (\gamma^{-1}\cdot\tau)$, when restricting to the action
of $\Gamma_\sigma$. Thus, we have

\begin{equation}
\label{first_eq}
[\sigma:\tau] = \sum_{\tau'\in\Gamma_\sigma\cdot\tau} \eta(\tau,\tau')\varepsilon(\sigma,\tau') + \sum_{\tilde\tau\in\Gamma_\sigma\cdot(\gamma^{-1}\cdot\tau)} \eta(\tau,\tilde\tau)\varepsilon(\sigma,\tilde\tau).    
\end{equation}
Using (i) and our choice of orientations, we get
\begin{equation}
\label{second_equation}
%(\ref{first_eq}) 
[\sigma:\tau]= |\Gamma_\sigma\cdot\tau| + |\Gamma_\sigma \cdot(\gamma^{-1}\cdot\tau)|\hspace{1pt} \varepsilon(\sigma,\gamma^{-1}\cdot\tau).    
\end{equation}
Using the same argument as in Lemma \ref{stab_groups} (ii), we get that $|\Gamma_\sigma\cdot\tau| = \frac{|\Gamma_\sigma|}{|\Gamma_\sigma\cap\Gamma_{\gamma\cdot\sigma}|}$ and $|\Gamma_\sigma\cdot\tau| = \frac{|\Gamma_\sigma|}{|\Gamma_\sigma\cap\Gamma_{\gamma^{-1}\cdot\sigma}|}$. As $|\Gamma_\sigma\cap\Gamma_{\gamma\cdot\sigma}| = |\Gamma_\sigma\cap\Gamma_{\gamma^{-1}\cdot\sigma}|$, then $|\Gamma_\sigma\cdot\tau|= |\Gamma_\sigma\cdot(\gamma^{-1}\cdot\tau)|$.
Hence,
\begin{equation*}
%(\ref{second_equation}) 
[\sigma:\tau]= |\Gamma_\sigma\cdot\tau|\left( 1 + \varepsilon(\sigma,\gamma^{-1}\cdot\tau)\right). 
\end{equation*}

Thus, it is enough to prove that $\varepsilon(\sigma,\gamma^{-1}\cdot\tau) = -1$. Consider $\mathcal{B}$ a positively oriented basis of $\mathbb{R}(\tau)$. By our choice of orientations, $\gamma^{-1}\cdot\mathcal{B}$ is a positively oriented basis of $\mathbb{R}(\gamma^{-1}\cdot\tau)$. By Proposition~\ref{deter_action} and Lemma~\ref{orientaciones_opuestas}, we conclude that $\varepsilon(\sigma,\gamma^{-1}\cdot\tau) = -1$.
\qedhere
\end{enumerate}
\end{proof}
\end{lemma}

\begin{theorem}
\label{top_cycle}
The element of $\operatorname{Vor}_{\Gamma}$
$$\sum_{\sigma\in\Sigma_{d(n)}} \frac{1}{|\Gamma_\sigma|} \sigma$$ is an explicit canonical non trivial $d(n)$-cycle and $H_{d(n)}\left(\operatorname{Vor}_{\Gamma}\otimes\mathbb{Q}\right)\cong\mathbb{Q}$.
\begin{proof}
Note that $H_{d(n)}\left(\operatorname{Vor}_{\Gamma}\otimes\mathbb{Q}\right)=\ker(d_{d(n)})$. 

Consider $\alpha = \sum_{\sigma\in\mathcal{S}\subset\rcells} \lambda_\sigma \sigma$ where $\lambda_\sigma\neq 0$ for all $\sigma\in\mathcal{S}$. Suppose that $d_{d(n)}(\alpha)=0$. By Lemma \ref{expresion_diferential} (iii), we can just consider $\tau\in\nrfacets$.

If $|\rcells| = 1$, then $\nrfacets = \emptyset$ and we are done. Thus, suppose that $|\rcells| > 1$.

By hypothesis, $\mathcal{S}\neq\emptyset$. Let $\sigma\in\mathcal{S}$. Since we are assuming that $|\rcells| > 1$, it holds that $\nfacets\neq\emptyset$.  By Lemma \ref{orientation_SLnZ} (ii) and Lemma~\ref{cod1_no_interseca_frontera}, for every $\tau'\in\nfacets$ there exists $\tau\in\nrfacets$ such that $\tau'$ is equivalent to $\tau$ under the action of $\Gamma$.

Take $\gamma_1,\gamma_2\in\Gamma$ and $\rho\in\rcells\setminus\{\sigma\}$ such that $\tau = \gamma_1\cdot\sigma\cap\gamma_2\cdot\rho$. Note that by our choice of $\rfacets$ we know that $\gamma_1 = 1$ or $\gamma_2 = 1$. By Lemma \ref{expresion_diferential} (ii) we have $$\frac{\lvert [\sigma : \tau ]\rvert}{|\Gamma_\sigma|} = \frac{\lvert [\rho:\tau] \rvert}{|\Gamma_\rho|} = \frac{1}{|\Gamma_\tau|}$$ and $$\operatorname{sign}([\sigma:\tau]) = -\operatorname{sign}([\rho:\tau]).$$ 

Moreover, it holds that $[\sigma':\tau] = 0$ for every $\sigma' \in \rcells\setminus\{ \sigma,\rho\}$. Therefore, as we are assuming $d_{d(n)}(\alpha)=0$, necessarily $\rho\in\mathcal{S}$ and $\lambda_\sigma = \frac{\lambda}{|\Gamma_\sigma|}$ and $\lambda_\rho = \frac{\lambda}{|\Gamma_\rho|}$ for $\lambda\in\mathbb{Q}$. 

By Theorem~\ref{vor_graph_connected}, we conclude that $\mathcal{S} = \rcells$ and that $$\alpha = \lambda \left(\sum_{\sigma\in\Sigma_{d(n)}} \frac{1}{|\Gamma_\sigma|} \sigma\right).$$
Without loss of generality suppose that $\lambda = 1$. 
We have
\begin{align*}
d(\alpha)
&= d\left(\sum_{\sigma \in \Sigma_{d(n)}} \frac{1}{|\Gamma_\sigma|}\,\sigma\right)
= \sum_{\sigma \in \Sigma_{d(n)}} \frac{1}{|\Gamma_\sigma|}\,d(\sigma) \\
&= \sum_{\sigma \in \Sigma_{d(n)}} \frac{1}{|\Gamma_\sigma|}
   \sum_{\tau \in \nrfacets} [\sigma : \tau]\,\tau .
\end{align*}

We can rewrite the previous equation as
$$d(\alpha) = \sum_{\tau\in\nrfacets}\left(\frac{1}{|\Gamma_{\sigma_\tau}|}[\sigma_\tau:\tau] + \frac{1}{|\Gamma_{\rho_\tau}|}[\rho_\tau:\tau] \right)\tau,$$
where $\tau = \sigma_{\tau}\cap\gamma\cdot\rho_{\tau}$, for $\gamma\in \Gamma$.
From the above discussion we deduce that:
$$d(\alpha) = \sum_{\tau\in\nrfacets}\left(\frac{1}{|\Gamma_{\sigma_\tau}|}\frac{|\Gamma_{\sigma_{\tau}}|}{|\Gamma_{\tau}|} - \frac{1}{|\Gamma_{\rho_\tau}|}\frac{|\Gamma_{\rho_{\tau}}|}{|\Gamma_{\tau}|} \right)\tau = 0.$$

As $H_{d(n)}\left(\operatorname{Vor}_{\Gamma}\otimes\mathbb{Q}\right)\cong\ker(d_{d(n)})$, we conclude that $\alpha\in H_{d(n)}\left(\operatorname{Vor}_{\Gamma}\otimes\mathbb{Q}\right)$.
\end{proof}
\end{theorem}

\begin{example}
Consider $\Gamma = \SL$. Lemma~\ref{expresion_diferential} gives an explicit description of the matrix associated to $d_{d(n)}$. Let $n_i = |\Sigma_{d(n)-i}|$. Let $A$ be the $n_1\times n_0$ matrix associated to $d_{d(n)}$. Every element in $\srfacets$ corresponds to a zero row of $A$. If a row $i$ corresponds to an element in $\nrfacets$, then it has exactly two nonzero entries $(i,j)$, and
\[
\left|A_{ij}\right| = \frac{|\Gamma_{\sigma_{j}}|}{|\Gamma_{\tau_i}|}.
\]

For $n = 2,3$, Corollary~\ref{cor} implies that $\Sigma_{d(n)-1} = \emptyset$. For $n = 4,5,6$, we have $\srfacets = \emptyset$, hence $A$ has no zero rows, see \cite[Section 5.1]{philippe_advances}.

The case $n = 7$ is the first for which $\srfacets \neq \emptyset$. In particular, $|\srfacets| = 5$, so $A$ has $5$ zero rows.

This is reflected as follows. In \cite[Table 7.9.3]{euclidean_lattices_martinet}, for each $\sigma \in \Sigma^*_{d(n)}$ the number of $\Gamma_\sigma$-orbits of $\facets$ is given. The diagonal entries correspond to the number of $\Gamma_\sigma$-orbits of elements in $\sfacets$, and their sum is $57$. On the other hand, \cite{philippe_advances} gives $|\starsrfacets| = 52$. Hence, there are exactly $5$ $\Gamma$-orbits that split into two orbits under the action of the stabilizer of the top cell containing them. This corresponds to Lemma~\ref{self_intersecting_facets_stab}~iv).
\end{example}

\begin{remark}
In \cite{Sharbly}, the authors use the Sharbly complex $Sh_\bullet$ to construct an explicit cycle in $H_{d(n)}(Sh_\bullet \otimes_\Gamma \mathbb{Q})$, where $\Gamma$ is a finite-index subgroup of $\SL$. One has
$$
H_{d(n)}(Sh_\bullet \otimes_\Gamma \mathbb{Q})
\cong
H_{d(n)}(\Gamma,\mathrm{St}_n).
$$
Given a perfect triangulation $\mathcal{S}$ of $\Sigma_{d(n)}$, they construct a specific family of sharblies such that
$$
z_\Gamma = \sum_{s\in\mathcal{S}} [s]_\Gamma + \sum_{\alpha} [y_\alpha]_\Gamma
$$
is a nontrivial cycle. The argument relies on properties of regular triangulations of polytopes and on the Voronoi tessellation

For $n=2,3$, where there is only one simplicial perfect form $\mathbb{A}_n$, $z_\Gamma$ coincides with the top cycle given in Theorem~\ref{top_cycle}.

In contrast, the Voronoi complex satisfies $V_k=0$ for $k>d(n)$, whereas the Sharbly complex has non-vanishing groups in higher degrees. Therefore, in the Sharbly setting the non-triviality of $z_\Gamma$ requires additional arguments. 

Moreover, the cycle $z_\Gamma$ depends on a series of auxiliary choices, while in the Voronoi complex the top cycle is canonically determined by the structural rigidity of the Voronoi tessellation. Another crucial difference is that the Voronoi complex is finite, whereas $Sh_\bullet \otimes_\Gamma \mathbb{Q}$ is infinite.
\end{remark}

As a consequence of the proof of Theorem~\ref{top_cycle}, we can prove the following statement:
\begin{corollary}
\label{no_top_cycle}
Let $n\in\mathbb{N}$, $n>0$ and  even. We have  $$H_{d(n)}(\operatorname{Vor}_{\GL}\otimes \mathbb{Q}) = 0.$$
\begin{proof}
By Proposition~\ref{deter_action}, we know that $\sigma\in\Sigma_{d(n)}$ if and only if $\Gamma_\sigma\subset\SL$. For instance, this is not the case for the family of root lattices, as the Weyl groups is a subset of its automorphism group. Thus, $\Sigma_{d(n)}\subsetneq \Sigma_{d(n)}^*$. 

Suppose that there exists $\sigma\in\Sigma_{d(n)}$, if not the statement follows trivially. In particular, $\sigma$ is the Voronoi Domain of a non-root lattice. Therefore, we can assume that $n\geq 6$. Then, by Theorem ~\ref{vor_graph_connected},  $|\nfacets|>0$. By Lemma~\ref{stab_groups}, every $\tau\in\nfacets$ is equivalent to an element in $\nrfacets$. By the argument in the proof of Theorem~\ref{top_cycle}, if there exits $\alpha\in V_{d(n)}$ such that $d_{d(n)}(\alpha) = 0$, then $\alpha = 0$.
\end{proof}
\end{corollary}
\begin{remark} Through Borel-Serre duality (and the equivariant spectral sequence associated to the Voronoi complex), Theorem \ref{top_cycle} recovers that $H^0(\Gamma,\mathbb{Q})\cong \mathbb{Q}$ when $\Gamma$ is a finite index subgroup of $\SL$. This is also the case for finite index subgroups of $\GLo$ in the Hermitian case, see \cite[Theorem 3.7]{hermitian}. On the other hand, when $\Gamma=\GL$, Corollary~\ref{no_top_cycle}  recovers that $H^0(\GL, \tilde{\mathbb{Z}})\otimes \mathbb{Q}=0$. Where $\tilde{\mathbb{Z}}$ is an orientation module associated to the action of $\GL$ on the cone of symmetric quadratic form over $\mathbb{R}$, see \cite[Section 7.2]{philippe_advances} for further details. 
\end{remark}
\section*{Acknowledgments}
\thanks{\small The author is fully supported by the COGENT project which has received funding from the European Union's Horizon Europe Programme under the Marie Sklodowska-Curie actions HORIZON-MSCA-2023-DN-01 call (Grant agreement ID: 101169527), and from UK Research and Innovation. The author thanks Philippe Elbaz-Vincent for introducing him to this problem, for his guidance, and for helpful discussions throughout the development of this work, and Gabriel Jalil for valuable discussions.}
\newpage
\bibliographystyle{plain}  
\bibliography{referencias}    

@article{Jaquet_rank_7,
     author = {Jaquet-Chiffelle, David-Olivier},
     title = {\'{E}num\'eration compl\`ete des classes de formes parfaites en dimension 7},
     journal = {Annales de l'Institut Fourier},
     pages = {21--55},
     year = {1993},
     publisher = {Institut Fourier},
     address = {Grenoble},
     volume = {43},
     number = {1},
     doi = {10.5802/aif.1320},
     mrnumber = {94d:11048},
     zbl = {0769.11028},
     language = {fr},
     url = {https://www.numdam.org/articles/10.5802/aif.1320/}
}

@article{philippe_advances,
title = {Perfect forms, {K}-theory and the cohomology of modular groups},
journal = {Advances in Mathematics},
volume = {245},
pages = {587-624},
year = {2013},
issn = {0001-8708},
doi = {https://doi.org/10.1016/j.aim.2013.06.014},
url = {https://www.sciencedirect.com/science/article/pii/S0001870813002223},
author = {Philippe Elbaz-Vincent and Herbert Gangl and Christophe Soulé}
}

@article{Voronoi1908,
author = {Voronoi, Georges},
journal = {Journal für die reine und angewandte Mathematik},
pages = {198-287},
title = {Nouvelles applications des paramètres continus à la théorie des formes quadratiques. Deuxième mémoire. Recherches sur les parallélloèdres primitifs.},
url = {http://eudml.org/doc/149291},
volume = {134},
year = {1908},
}

@book{euclidean_lattices_martinet,
author = {Martinet, Jacques},
year = {2003},
month = {01},
pages = {},
title = {Perfect Lattices in Euclidean Spaces},
volume = {327},
isbn = {978-3-642-07921-4},
doi = {10.1007/978-3-662-05167-2}
}

@article{hermitian,
title = {On the cohomology of linear groups over imaginary quadratic fields},
journal = {Journal of Pure and Applied Algebra},
volume = {220},
number = {7},
pages = {2564-2589},
year = {2016},
issn = {0022-4049},
doi = {https://doi.org/10.1016/j.jpaa.2015.12.002},
url = {https://www.sciencedirect.com/science/article/pii/S002240491500345X},
author = {Mathieu {Dutour Sikirić} and Herbert Gangl and Paul E. Gunnells and Jonathan Hanke and Achill Schürmann and Dan Yasaki}
}

@article{2002,
  title = {Some computations of the homology of {GL}$_N(\mathbb{Z})$ and the {K}-theory of $\mathbb{Z}$},
  author={Elbaz-Vincent, P and Gangl, H and Soule, C},
  journal={Comptes Rendus Mathematique},
  volume={335},
  number={4},
  pages={321--324},
  year={2002},
  publisher={EDITIONS SCIENTIFIQUES MEDICALES ELSEVIER 23 RUE LINOIS, 75724 PARIS CEDEX~…}
}

@article{Bordeaux,
     author = {Kenji Okuda and Syouji Yano},
     title = {A generalization of {Vorono{\"\i}{\textquoteright}s} {Theorem} to algebraic lattices},
     journal = {Journal de th\'eorie des nombres de Bordeaux},
     pages = {727--740},
     year = {2010},
     publisher = {Universit\'e Bordeaux 1},
     volume = {22},
     number = {3},
     doi = {10.5802/jtnb.742},
     zbl = {1253.11072},
     mrnumber = {2769341},
     language = {en},
     url = {https://jtnb.centre-mersenne.org/articles/10.5802/jtnb.742/}
}

@article{K8, 
    title = {Voronoi complexes in higher dimensions, cohomology of {GL}$_{N}(\mathbb{Z})$ for $N \geq 8$ and the triviality of {K}$_8(\mathbb{Z})$},
    volume={25},
    DOI={10.1017/S1474748025101394}, 
    number={1}, 
    journal={Journal of the Institute of Mathematics of Jussieu}, 
    author={Dutour Sikirić, Mathieu and Elbaz-Vincent, Philippe and Kupers, Alexander and Martinet, Jacques}, 
    year={2026}, 
    pages={565–590}
}

@article{Sharbly,
  author    = {Avner Ash and Paul E. Gunnells and Mark McConnell},
  title     = {Explicit sharbly cycles at the virtual cohomological dimension for $\textrm{SL}_n(\mathbb{Z})$},
  journal   = {Journal of Homotopy and Related Structures},
  year      = {2025},
  volume    = {20},
  number    = {3},
  pages     = {391--416},
  doi       = {10.1007/s40062-025-00374-9},
  url       = {https://doi.org/10.1007/s40062-025-00374-9},
  issn      = {1512-2891}
}

@article{Borel,
  author = {Borel, A. and Serre, J.-P.},
  title = {Corners and Arithmetic Groups},
  journal = {Commentarii Mathematici Helvetici},
  volume = {48},
  year = {1973},
  pages = {436--483},
  url = {http://eudml.org/doc/139559}
}

@article{Venkatesh2016,
  author    = {Akshay Venkatesh},
  title     = {Cohomology of arithmetic groups and periods of automorphic forms},
  journal   = {Japanese Journal of Mathematics},
  year      = {2017},
  volume    = {12},
  number    = {1},
  pages     = {1--32},
  doi       = {10.1007/s11537-016-1488-2},
  url       = {https://doi.org/10.1007/s11537-016-1488-2},
  issn      = {1861-3624}
}

@article{VenkateshICM2018,
author = {Akshay Venkatesh},
title = {Cohomology of Arithmetic Groups- Fields Medal Lecture},
journal = {Proceedings of the International Congress of Mathematicians (ICM 2018)},
chapter = {},
pages = {267-300},
doi = {10.1142/9789813272880_0014},
URL = {https://www.worldscientific.com/doi/abs/10.1142/9789813272880_0014},
eprint = {https://www.worldscientific.com/doi/pdf/10.1142/9789813272880_0014}
}

@book{Harder2025,
  author    = {Günter Harder},
  title     = {Cohomology of Arithmetic Groups},
  publisher = {Springer},
  address   = {Cham},
  year      = {2026},
  series    = {Springer Monographs in Mathematics},
  doi       = {10.1007/978-3-032-10378-9},
  isbn      = {978-3-032-10378-9},
  url       = {https://doi.org/10.1007/978-3-032-10378-9}
}

@article{Soule2000,
title = {On the 3-torsion in {K}$_4(\mathbb{Z})$},
journal = {Topology},
volume = {39},
number = {2},
pages = {259-265},
year = {2000},
issn = {0040-9383},
doi = {https://doi.org/10.1016/S0040-9383(99)00006-3},
url = {https://www.sciencedirect.com/science/article/pii/S0040938399000063},
author = {C. Soulé}
}

@article{general_setting,
title = {Computing in arithmetic groups with {V}oronoï's algorithm},
journal = {Journal of Algebra},
volume = {435},
pages = {263-285},
year = {2015},
issn = {0021-8693},
doi = {https://doi.org/10.1016/j.jalgebra.2015.01.022},
url = {https://www.sciencedirect.com/science/article/pii/S0021869315000617},
author = {Oliver Braun and Renaud Coulangeon and Gabriele Nebe and Sebastian Schönnenbeck}
}

\end{document}